\documentclass[11pt]{article}
\usepackage{pdfpages}
\usepackage{multirow}
\usepackage{array}
\usepackage{graphicx}
\usepackage{subcaption}
\usepackage{latexsym}
\usepackage{pstricks}
\usepackage{pst-node}
\usepackage{pst-tree}
\usepackage{url}
\usepackage{times}
\usepackage{helvet}
\usepackage{courier}
\usepackage{algorithmic}
\usepackage[ruled,vlined,linesnumbered]{algorithm2e}
\usepackage{amsthm}
\usepackage{graphicx,amssymb,amsmath}
\usepackage{mathrsfs}
\usepackage{mathtools}
\usepackage{comment}

\usepackage[hypertexnames=false, colorlinks=true, citecolor=blue, linkcolor=red, urlcolor=black]{hyperref}
\usepackage{geometry}
\usepackage{authblk}

\makeatletter
\renewcommand*{\verbatim@font}{\sffamily}

\title{Lower Bounds for Small Ramsey Numbers on Hypergraphs\thanks{A preliminary version of this paper appeared in the proceedings of COCOON 2019.}}

\author{S. Cliff Liu \\
Princeton University	\\
\textsf{sixuel@cs.princeton.edu}
}

\begin{document}

\maketitle


\newcounter{dummy} \numberwithin{dummy}{section}
\newtheorem{lemma}[dummy]{Lemma}
\newtheorem{definition}[dummy]{Definition}
\newtheorem{remark}[dummy]{Remark}
\newtheorem{corollary}[dummy]{Corollary}
\newtheorem{claim}[dummy]{Claim}
\newtheorem{observation}[dummy]{Observation}
\newtheorem{conjecture}[dummy]{Conjecture}

\newtheorem{theorem}{Theorem}

\renewcommand{\algorithmicrequire}{\textbf{Input:}}
\renewcommand{\algorithmicensure}{\textbf{Output:}}

\begin{abstract}
    The Ramsey number $r_k(p, q)$ is the smallest integer $N$ that satisfies for every red-blue coloring on $k$-subsets of $[N]$, there exist $p$ integers such that any $k$-subset of them is red, or $q$ integers such that any $k$-subset of them is blue. In this paper, we study the lower bounds for small Ramsey numbers on hypergraphs by constructing counter-examples and recurrence relations. We present a new algorithm to prove lower bounds for $r_k(k+1, k+1)$. In particular, our algorithm is able to prove $r_5(6,6) \ge 72$, where there is only trivial lower bound on $5$-hypergraphs before this work. We also provide several recurrence relations to calculate lower bounds based on lower bound values on smaller $p$ and $q$. Combining both of them, we achieve new lower bounds for $r_k(p, q)$ on arbitrary $p$, $q$, and $k \ge 4$.
\end{abstract}

\section{Introduction}\label{intro}
At least how many guests you have to invite for a party to make sure there are either certain number of people know each other or certain number of people do not know each other? The answer is the classical Ramsey number.
Ramsey theory generally concerns unavoidable structures in graphs, and has been extensively studied for a long time \cite{erdos1956partition, shelah1988primitive, DBLP:conf/sat/HeuleKM16}.
However, determining the exact Ramsey number is a notoriously difficult problem, even for small $p$ and $q$.
For example, it is only known that the value of $r_2(5, 5)$ is between $43$ to $48$ inclusively, and for $r_2(10, 10)$, people merely know a much rougher range from $798$ to $23556$ \cite{DBLP:journals/jct/McKayR97, DBLP:journals/jct/Shearer86, DBLP:journals/dm/Shi03}.

As for the hypergraph case of $k \ge 3$, our understanding of Ramsey number is even less.
The only known exact value of Ramsey number is $r_3(4, 4) = 13$, with only loose lower bounds for other values of $p$, $q$, and $k$ \cite{DBLP:conf/soda/McKayR91, radziszowski1994small}.
Although some progresses have been made for $r_4(p, q)$, and particularly, lower bound for $r_4(5, 5)$ has been continuously pushed forward in the past thirty years, the recurrence relations remain the same, i.e., one can immediately obtain better lower bounds for $p, q \ge 6$ by substituting into improved bound for $r_4(5, 5)$, but there is no other way to push them further \cite{DBLP:journals/dm/Shastri90, DBLP:journals/dm/SongYL95}.

Another fruitful subject in Ramsey theory is the asymptotic order of Ramsey number. Using the so-called Stepping-up Lemma introduced by Erd\H{o}s and Hajnal, the Ramsey number $r_k(p, n)$ is lower bounded by the tower function $t_k(c \cdot f(n))$ defined by $t_1(x) = x, t_{i+1}(x) = 2^{t_i(x)}$, where $f(n)$ is some function on $n$ and $c$ is a constant depending on $p$ \cite{erdHos1965partition, graham1990ramsey}.
Recent research improves the orders of $r_4(5, n)$ and $r_4(6, n)$ and leads to similar bounds for $r_k(k+1, n)$ and $r_k(k+2,n)$ \cite{conlon2010hypergraph}.
We point out that their lower bounds for $r_k$ depends on $r_{k-1}$. In other words, to get a lower bound for $r_k(p, q)$, one must provide the lower bounds for some $r_{k-1}(p',q')$. More importantly, when focusing on Ramsey numbers on small $p,q$ values, the Stepping-up Lemma cannot be applied directly.
We refer readers to Chapter 4.7 in \cite{graham1990ramsey} for details.

It is well known that directly improving the lower bounds for Ramsey number is extremely hard, since it requires tremendous computing resources \cite{gaitan2012ramsey}. A possible method to attack this is to use recurrence relations based on the initial values.
However, calculating a good initial value itself can be way beyond our reach. 
For instance, a simple attempt to push the current best lower bound $r_2(6,6) \ge 102$ could be constructing a CNF (Conjunctive Normal Form) whose satisfying assignment is equivalent to a $6$-clique free and $6$-independent-set free graph on $102$ vertices. This CNF has size (the number of literals in the formula) about ${10}^{10}$, but state-of-the-art SAT solvers are only capable of solving CNF with size no more than ${10}^{6}$, and is almost sure to not terminate in reasonable time \cite{DBLP:conf/sat/TompkinsH04a, DBLP:series/faia/2009-185}. 

\paragraph{Contributions.}
We prove several recurrence relations in the form of $r_k(p,q) \ge d\cdot (r_k(p-1,q) - 1) + 1$, where $d$ depends on $p$, $q$, and $k$.
Two of them are for arbitrary integer $k \ge 4$. To the best of our knowledge, this is the first recurrence relation on $r_k(p,q)$ not depending on $r_{k-1}(p,q)$, but for arbitrary $k$.
To build our proof, we introduce a method called \emph{pasting}, which constructs a good coloring by combining colorings on smaller graphs.
The recurrence relations are proven by inductions, where several base cases are proven by transforming to an equivalent CNF and solved by a SAT solver.
Additionally, to obtain a good initial values of the recurrence relations, a new algorithm for constructing counter-example hypergraphs is proposed, which efficiently proves a series of lower bounds for Ramsey number on $k$-hypergraphs including $r_5(6,6) \ge 72$: the first non-trivial result of lower bounds on $5$-hypergraphs.
The algorithm is based on local search and is easy to implement.
Combining both techniques, we significantly improve the lower bounds for $r_4(p, q)$ and achieve new non-trivial lower bounds for $r_k(p, q)$ on arbitrary $p$, $q$, and $k \ge 5$.

\paragraph{Roadmap.}
The rest of this paper is organized as follows.
In \S{\ref{preliminaries}} we introduce fundamental definitions.
The basic forms of recurrence relations are given in \S{\ref{forms}}.
In \S{\ref{smallk}} we present proofs for the recurrence relations on several small values of $k$, followed by two recurrence relations on arbitrary $k$ in \S{\ref{bigk}}.
Finally, we summarize some of our new lower bounds in \S{\ref{lower_bounds}}.
The formal recurrence relations are given in Theorems \ref{k4}, \ref{k567_p-1}, \ref{k_p-1}, \ref{two_divide}, and \ref{large_q}.
Our algorithm for calculating lower bounds for $r_{k}(k+1, k+1)$ is presented in Appendix~\S{\ref{local_search}}.

\section{Preliminaries}\label{preliminaries}
In this section, basic notations in Ramsey theory are introduced, followed by a sketch of our proof procedure.
Then we propose our key definitions and several useful conclusions.

\subsection{Notations}
A $k$-uniform hypergraph $G(V, E; k)$ is a tuple of vertex set $V$ and a set $E$ of hyperedges such that each hyperedge in $E$ is a $k$-subset of $V$, where each $e \in E$ is called a $k$-hyperedge. If the context is clear, $G(V,E)$ or $G$ is used instead.
A complete $k$-uniform hypergraph consists of all possible $k$-subsets of $V$ as its hyperedge set.
We only deal with complete $k$-uniform hypergraphs and may use $k$-graph (or graph) and edge for short.
Given a vertex set $V$ with $|V| \ge k$, we use $V^{(k)}$ to denote the complete $k$-uniform hypergraph.

A coloring is a mapping $\chi^{(k)}: E \rightarrow \{\text{red}, \text{blue}\}$ that maps all $k$-hyperedges in $E$ to red or blue.
We write $\chi^{k}(e)=\text{red}$ for coloring some edge $e \in E$ with red under $\chi^k$.
Given $G(V, E; k)$, we say $\chi^{(k)}$ is a $(p,q;k)$-coloring of $G$ if there is neither red $p$-clique nor blue $q$-clique in $G$.
We also use $\chi$ instead of $\chi^{(k)}$ if there is no ambiguity.
A $p$-clique is a complete subgraph induced by $p$ vertices, and a red (resp. blue) $p$-clique is a clique where all edges are red (resp. blue).

The Ramsey number $r_k(p,q)$ is the minimum integer $N$ that satisfies there is no $(p,q;k)$-coloring for $G(V,E;k)$ on $|V| = N$ vertices. In other words, for any coloring on $G$, there is either a red $p$-clique or a blue $q$-clique.

\subsection{A Proof Procedure}\label{proof_procedure}
We prove recurrence relations in the form of $r_k(p,q) \ge d\cdot (r_k(p-1,q) - 1) + 1$ by the following procedure \textsf{Pasting}:
\begin{enumerate}
    \item Given integer $d$, for each $i \in [d]$, let $G_i(V_i, E_i)$ be a graph on $r_k(p-1,q) - 1$ vertices with $(p-1,q;k)$-coloring $\chi_i$.
    \item Add an edge for every $k$-subset of $\bigcup_{i \in [d]} V_i$ if there is no edge on it. Denote the set of added edges as $\mathfrak{E}$. Let the complete graph after adding all edges be $\mathbf{G}(\bigcup_{i \in [d]} V_i, (\bigcup_{i \in [d]} E_i ) \bigcup \mathfrak{E})$.
    \item Construct $\chi'$ on $\mathfrak{E}$ such that $\chi \coloneqq (\bigcup_{i \in [d]} \chi_i) \bigcup \chi'$ on $\mathbf{G}$ satisfies that each $p$-clique of $\bigcup_{i \in [d]} V_i$ contains a blue edge and each $q$-clique of it contains a red edge. \label{step_3}
    \item It can be concluded that $r_k(p,q) \ge d\cdot (r_k(p-1,q) - 1) + 1$ since $\chi$ is a $(p,q;k)$-coloring for $\mathbf{G}$ and $\left| \bigcup_{i \in [d]} V_i \right| = d\cdot (r_k(p-1,q) - 1)$.
\end{enumerate}
The non-trivial step in \textsf{Pasting} is Step \ref{step_3} (\emph{coloring construction}), which will be discussed in details in \S{\ref{smallk}} and \S{\ref{bigk}}.
\textsf{Pasting}$(k,p,q,d)$ is \emph{successful} if $\chi = (\bigcup_{i \in [d]} \chi_i) \bigcup \chi'$ can be found.

\subsection{Primal Cardinality Vector}

Observe that the coloring construction cannot depend on the order of $G_i$ dues to symmetry, thus a primal order shall be fixed and our coloring depends only on the sequence of cardinalities of the intersections in non-increasing order. We introduce the following concepts concerning this.

Let $V_1, V_2, \dots, V_d$ be $d$ disjoint sets each with cardinality $r_k(p-1,q)-1$, and let $V$ be $\bigcup_{i \in [d]} V_i$. For any $\sigma$-subset $X \subseteq V$, define cardinality vector $\mathbf{\hat{v}}(X) = (\hat{v}_1, \hat{v}_2, \dots, \hat{v}_d)$ where $\hat{v}_i = |X \bigcap V_i|$. Let $\hat{v}_{(1)}, \hat{v}_{(2)}, \dots, \hat{v}_{(d)}$ be the sequence after sorting the $\hat{v}_i$'s in a non-increasing order.
\begin{definition}\label{pcv}
    Given $V$, $X$, and $\{\hat{v}_{(i)} \mid i \in [d]\}$ as above,
    define \emph{primal cardinality vector} of $X$ as $\mathbf{v}(X) = (v_1, v_2, \dots, v_{\pi(X)})$, where $v_i = \hat{v}_{(i)}$ for all $i \in [\pi(X)]$, and $\pi(X)$ satisfies either (\romannumeral1) $\pi(X) = d$ or (\romannumeral2) $\hat{v}_{(\pi(X))} > 0$ and $\hat{v}_{(\pi(X) + 1)} = 0$.
\end{definition}
In a word, $\mathbf{v}(X)$ is a sequence of all positive coordinates of the cardinality vector $\mathbf{\hat{v}}(X)$ in a non-increasing order.
Observe that when $\sigma = |X| = k$, $X$ corresponds to some edge $e(X)$ in $\mathbf{G}$, and $\mathbf{v}(X) = (v_1, v_2, \dots, v_{\pi(X)})$ essentially means that $e(X)$ has $v_i$ endpoints in the $i$-th subgraph (in a non-increasing order of the cardinalities of intersections).
Usually primal cardinality vector $\mathbf{v}$ shows up without indicating which set $X$ it corresponds to, and we refer $\pi(\mathbf{v})$ to the length of $\mathbf{v}$.

The following remark captures the idea we proposed at the beginning of this subsection.
\begin{remark}
    In Step \ref{step_3} of \textsf{Pasting}, $\forall e_1,e_2 \in \mathfrak{E}$, $\chi'(e_1) = \chi'(e_2)$ if $\mathbf{v}(e_1) = \mathbf{v}(e_2)$.
\end{remark}

We will write $\mathbf{v}(e)$ instead of $\mathbf{v}(X)$ when $X$ corresponds to edge $e$.
In this case, abusing the notation slightly, we write $\chi(\mathbf{v}(e))$ as the color under $\chi$ on edge $e$, since all edges with the same primal cardinality vector $\mathbf{v}$ are in same color.
Furthermore, we write $\chi(\mathbf{v}) = c$ where $c$ is red or blue for assigning all edges with primal cardinality vector $\mathbf{v}$ to color $c$.
For any $i \in [\pi(X)]$, $v_i(X)$ is the $i$-th coordinate of $\mathbf{v}(X)$.

\begin{remark}{\label{remark2}}
    For any non-trivial $\sigma$-subset $X$ and any $\tau$-subset $Y \subseteq X$, it must be that:
    (\romannumeral1) $~\sum_{i \in [\pi(X)]} v_i(X) = \sigma > 0$,
    (\romannumeral2) $~\sum_{i \in [\pi(Y)]} v_i(Y) = \tau \le \sigma$, (\romannumeral3) $\pi(X) \ge \pi(Y)$,
    and (\romannumeral4) $\forall i \in [\pi(Y)]$, $v_i(X) \ge v_i(Y)$.
\end{remark}
\begin{proof}
    The first three bullets are simple cardinality properties. To show that property (\romannumeral4) holds, let $j$ be the smallest index with $v_j(Y) > v_j(X)$. If $j=1$, there is no way to fit the largest subset of $Y$ into any subset of $X$. Else if $j > 1$, the only way to fit $Y_j$ into a subset of $X$ is to swap it with some $Y_i~(i < j)$, but $v_i(Y) \ge v_j(Y) > v_i(X)$, then $Y_i$ cannot fit into the $i$-th subset of $X$.
 \end{proof}
\begin{definition}\label{partial_order}
    Given two primal cardinality vectors $\mathbf{v_1}$ and $\mathbf{v_2}$, define \emph{partial order} between them as: $\mathbf{v_1} \le_c \mathbf{v_2}$ if and only if
    (\romannumeral1) $\pi(\mathbf{v_1}) \le \pi(\mathbf{v_2})$
    and (\romannumeral2) $\forall i \in [\pi(\mathbf{v_1})]$, $\mathbf{v_1}_i \le \mathbf{v_2}_i$.
    If at least one of the inequalities in (\romannumeral1) and (\romannumeral2) is strict, then $\mathbf{v_1} <_c \mathbf{v_2}$.
    \footnote{$\mathbf{v_2} \ge_c \mathbf{v_1}$ reads ``$v_2$ contains $v_1$".
    }
\end{definition}
One can easily show that \emph{reflexivity}, \emph{antisymmetry} and \emph{transitivity} for any partial order hold for $\le_c$.
Under this definition, with Remark~{\ref{remark2}} and subsets enumeration we can immediately conclude the following:
\begin{corollary}\label{subset}
    Given $V = \bigcup_{i \in [d]} V_i$, $G = V^{(k)}$ and $X \subseteq V$, we have $\forall Y \subseteq X, \mathbf{v}(Y) \le_c \mathbf{v}(X)$,
    and $\forall \mathbf{v'} \le_c \mathbf{v}(X)$, $\exists Y \subseteq X$ with $\mathbf{v}(Y) = \mathbf{v'}$.
    Specifically, $Y$ corresponds to an edge $e(Y)$ of $G$ when $\sum_{i \in [\pi(\mathbf{v'})]} \mathbf{v'}_i = k$, and $e(Y)$ is an edge of $X^{(k)}$.
\end{corollary}

Given any subset of $V$,
observe that there are at most $d$ different subsets to be intersected with. As a result, given a $s$-subset, we only concern subsets with primal cardinality vectors in the following set:
\begin{definition}{\label{pcv_set}}
    Define $\mathbf{V_s}(d)$ as the set of all primal cardinality vectors $\mathbf{v}$ such that $\pi(\mathbf{v}) \le d$ and $\sum_{i \in [\pi(\mathbf{v})]} \mathbf{v}_i = s$.
\end{definition}
\begin{remark}\label{remark3}
    $\forall d \ge s, \mathbf{V_s}(d) = \mathbf{V_s}(s)$.
\end{remark}

Based on Corollary~\ref{subset}, we conclude this section with the following corollary:
\begin{corollary}{\label{partition_1}}
    Given integers $p$, $q$, $k$, $d$, $V = \bigcup_{i \in [d]} V_i$, and $G = V^{(k)}$, the following four statements are equivalent:
    \begin{enumerate}
        \item $\exists \chi$ such that $\forall \mathbf{v} \in \mathbf{V_p}(d)$ (resp. $\forall \mathbf{v} \in \mathbf{V_q}(d)$ ), $\exists \mathbf{v'} \in \mathbf{V_k}(d)$ such that $\mathbf{v'} \le_c \mathbf{v}$ and $\chi(\mathbf{v'})= \text{blue}$ (resp. red). \label{statement_1}
        \item $\exists \chi$ such that $\forall p$-subset (resp. $q$-subset) $X \subseteq V$, $\exists k$-hyperedge $e$ of $X^{(k)}$ such that $\chi(e)= \text{blue}$ (resp. red).
        \item \emph{\textsf{Pasting}}$(k,p,q,d)$ is successful.
        \item $r_k(p,q) \ge d\cdot (r_k(p-1,q) - 1) + 1$.
    \end{enumerate}
\end{corollary}

\section{Forms of Recurrences}\label{forms}
We prove $r_k(p,q) \ge d\cdot (r_k(p-1,q) - 1) + 1$ for three different forms of $d$ under different conditions:
(1) $d = 2$,
(2) $d = p - 1$,
and (3) $d = \lfloor \frac{q-1}{k-2} \rfloor$.
Form (3) requires the strongest condition but its proof turns out to be simpler.
For forms (1) and (2), we show that to prove recurrence relation on given $k$ and arbitrary $p, q$, it is sufficient to prove the base case on $p$ and $q$, i.e., prove the case on $p = p_0, q = q_0$ for some constants $p_0, q_0$.

Firstly we show that for a given integer $d$, if $r_k(p_0,q_0) \ge d\cdot (r_k(p_0-1,q_0) - 1) + 1$ is given by \textsf{Pasting}, then $r_k(p, q) \ge d \cdot (r_k(p-1,q) - 1) + 1$.
\begin{lemma}{\label{constant}}
    Given integer $d$,
    if \emph{\textsf{Pasting}}$(k,p_0,q_0,d)$ is successful, then $\forall p \ge p_0, q \ge q_0$, \emph{\textsf{Pasting}}$(k,p,q,d)$ is successful, which is $r_k(p,q) \ge d\cdot (r_k(p-1,q) - 1) + 1$.
\end{lemma}
\begin{proof}
    The proof relies on Corollary~\ref{partition_1}.
    Let $\chi_0$ be a $(p_0,q_0;k)$-coloring fed to \textit{\emph{\textsf{Pasting}}}.
    We have $\forall \mathbf{v} \in \mathbf{V_{p_0}}(d)$, $\exists \mathbf{v'} \in \mathbf{V_k}(d)$ such that $\mathbf{v'} \le_c \mathbf{v}$ and $\chi_0(\mathbf{v'})=$ blue.
    Meanwhile, by Corollary~\ref{subset} we know that $\forall p \ge p_0, \forall \mathbf{u} \in \mathbf{V_p}(d)$, $\exists \mathbf{v} \in \mathbf{V_{p_0}}(d)$ such that $\mathbf{v} \le_c \mathbf{u}$.
    By transitivity it must be that $\mathbf{v'} \le_c \mathbf{u}$.
    This means $\forall p \ge p_0, \forall \mathbf{u} \in \mathbf{V_p}(d)$,
    $\exists \mathbf{v'} \in \mathbf{V_k}(d)$ such that $\mathbf{v'} \le_c \mathbf{u}$ and $\chi_0(\mathbf{v'})= \text{blue}$.
    Using the same reasoning we get $\forall q \ge q_0, \forall \mathbf{u} \in \mathbf{V_q}(d)$,
    $\exists \mathbf{v'} \in \mathbf{V_k}(d)$ such that $\mathbf{v'} \le_c \mathbf{v}$ and $\chi_0(\mathbf{v'})= \text{red}$, and the conclusion follows.
 \end{proof}

Secondly we give the following lemma showing that the induction from the base case to arbitrary $p,q$ also holds for form (2).
\begin{lemma}{\label{dp}}
    Given $p_0 \ge q_0 + 1$, if \emph{\textsf{Pasting}}$(k,p_0,q_0,p_0-1)$ is successful, 
    then $\forall p \ge p_0, q \ge q_0$, \emph{\textsf{Pasting}}$(k,p,q,p-1)$ is successful, which is $r_k(p,q) \ge (p-1) \cdot (r_k(p-1,q) - 1) + 1$.
\end{lemma}
We give the sketch of the proof here, followed by two lemmas to integrate the formal proof.
\begin{proof}[Proof sketch of Lemma~\ref{dp}]
    The proof contains two parts.
    First, we need to show that
    \textsf{Pasting}$(k,p_0,q_0,p_0-1)$ is successful
    implies that
    $\forall p \ge p_0$, \textsf{Pasting}$(k,p,q_0,p-1)$ is successful.
    Then we prove that for arbitrary fixed $p$,
    $\forall q \ge q_0$, \textsf{Pasting}$(k,p,q,p-1)$ is successful.
    Combining both of these we can conclude the proof.
 \end{proof}
\begin{lemma}{\label{any_p}}
    Given $p_0 \ge q_0 + 1$, if \emph{\textsf{Pasting}}$(k,p_0,q_0,p_0-1)$ is successful, then
    $\forall p \ge p_0$, \emph{\textsf{Pasting}}$(k,p,q_0,p-1)$ is successful.
\end{lemma}
\begin{proof}
    By Corollary~\ref{partition_1}, if \textsf{Pasting}$(k,p_0,q_0,p_0-1)$ is successful, we have that $\exists \chi_0$ such that the following two statements hold:
    \begin{align}
        & \forall \mathbf{v} \in \mathbf{V_{p_0}}(p_0 - 1)~\exists \mathbf{v'} \in \mathbf{V_k}(p_0 - 1),~\mathbf{v'} \le_c \mathbf{v} \wedge \chi_0(\mathbf{v'})= \text{blue} \label{p1}. \\
        & \forall \mathbf{v} \in \mathbf{V_{q_0}}(p_0 - 1)~\exists \mathbf{v'} \in \mathbf{V_k}(p_0 - 1),~\mathbf{v'} \le_c \mathbf{v} \wedge \chi_0(\mathbf{v'})= \text{red} \label{q1}.
    \end{align}

    By induction on $p$, it remains to prove the inductive step:
    \textsf{Pasting}$(k,p_0 + 1,q_0,p_0)$ is successful, which is equivalent to that $\exists \chi_1$ such that the following two statements hold:
    \begin{align}
        & \forall \mathbf{v} \in \mathbf{V_{p_0 + 1}}(p_0)~\exists \mathbf{v'} \in \mathbf{V_k}(p_0),~\mathbf{v'} \le_c \mathbf{v} \wedge \chi_1(\mathbf{v'})= \text{blue} \label{p2}.\\
        & \forall \mathbf{v} \in \mathbf{V_{q_0}}(p_0)~\exists \mathbf{v'} \in \mathbf{V_k}(p_0),~\mathbf{v'} \le_c \mathbf{v} \wedge \chi_1(\mathbf{v'})= \text{red} \label{q2}.
    \end{align}

    We prove that any $\chi_0$ satisfies (\ref{p1}) and (\ref{q1}) also satisfies (\ref{p2}) and (\ref{q2}). First we prove that (\ref{p1}) implies (\ref{p2}), then we prove that (\ref{q1}) implies (\ref{q2}).
    Noticing that $p_0, q_0 \ge k+1$, otherwise the hypergraph is trivial. So by Remark~\ref{remark3} we know that $\mathbf{V_k}(p_0 - 1) = \mathbf{V_k}(p_0) = \mathbf{V_k}(k)$.

    For the first implication, by Definition~\ref{pcv_set},
    $\mathbf{V_{p_0 + 1}}(p_0) = \mathbf{V_{p_0 + 1}}(p_0 - 1) \bigcup \mathbf{V'}$ where $\mathbf{V'} = \{\mathbf{v} \mid \pi(\mathbf{v}) = p_0,~\sum_{i \in [\pi(\mathbf{v})]} v_i = p_0 + 1\} = \mathbf{1}^{p_0} + \mathbf{e_1}$.
    \footnote{Conventionally, $\mathbf{1}^n$ is a vector of length $n$ with all coordinates being 1; $\mathbf{e_i}$ is a vector with the $i$-th coordinate being $1$ and others being $0$.}
    Thus $\forall \mathbf{v} \in \mathbf{V_{p_0 + 1}}(p_0)$, there are two cases:
    (\romannumeral1) if $\mathbf{v} \in \mathbf{V_{p_0 + 1}}(p_0 - 1)$, then let $\mathbf{u} \coloneqq \mathbf{v} - \mathbf{e_{\pi(\mathbf{v})}}$;
    (\romannumeral2) else if $\mathbf{v} = \mathbf{1}^{p_0} + \mathbf{e_1}$, let $\mathbf{u} \coloneqq \mathbf{1}^{p_0 - 1} + \mathbf{e_1}$.
    In either case we have $\mathbf{u} \in \mathbf{V_{p_0}}(p_0 - 1)$ and $\mathbf{u} \le_c \mathbf{v}$.
    Also, by (\ref{p1}) we know that $\exists \mathbf{v'} \in \mathbf{V_k}(k)$ such that $\mathbf{v'} \le_c \mathbf{u}$ and $\chi_0(\mathbf{v'})= \text{blue}$, so by transitivity $\mathbf{v'} \le_c \mathbf{u} \le_c \mathbf{v}$, we have that $\chi_0$ satisfies (\ref{p2}).
    For the second implication, since $p_0 \ge q_0 + 1$, by Remark~\ref{remark3} we have $\mathbf{V_{q_0}}(p_0) = \mathbf{V_{q_0}}(p_0 - 1) = \mathbf{V_{q_0}}(q_0)$, then (\ref{q1}) is equivalent to (\ref{q2}).

    As a result, $\chi_1$ satisfies (\ref{p2}) and (\ref{q2}), by which we finish the induction and conclude the proof.
 \end{proof}
\begin{lemma}{\label{any_q}}
    Given integers $p, q_0$, if \emph{\textsf{Pasting}}$(k,p,q_0,p-1)$ is successful, 
    then $\forall q \ge q_0$, \emph{\textsf{Pasting}}$(k,p,q,p-1)$ is successful.
\end{lemma}
\begin{proof}
    Since $p$ is fixed, by Lemma~\ref{constant} with $d=p-1$, we have that \textsf{Pasting}$(k,p',q_0,p-1)$ is successful for any $p' \ge p, q \ge q_0$. In particular, the conclusion holds for $p' = p$ and any $q \ge q_0$
 \end{proof}

By the proof sketch of Lemma~\ref{dp}, with Lemma~\ref{any_p} and \ref{any_q} we finish the proof of Lemma~\ref{dp}.

\section{Recurrences for Small $\mathbf{k}$}\label{smallk}
In this section, we give our main results on recurrence relations for small $k$, followed by their proofs and the relation to the satisfiability problem.
\subsection{Main Results on Small $k$}
\begin{theorem}{\label{k4}}
    For any integer $p \ge 6$ and $q \ge 5$, $r_4(p,q) \ge 2 r_4(p-1,q) - 1$ holds.
    Furthermore, if $q \ge 7$ then $r_4(p,q) \ge (p-1) \cdot (r_4(p-1,q)-1) + 1$ holds.
\end{theorem}

\begin{theorem}{\label{k567_p-1}}
    There exists a constant $c \ge 25$,
    such that given integer $k \ge 5$ and $k \le c$, for any integer $p \ge k+2$ and $q \ge k+2$, $r_k(p,q) \ge (p-1) \cdot (r_k(p-1,q)-1) + 1$ holds.
\end{theorem}

\begin{theorem}{\label{k_p-1}}
    There exists a constant $c \ge 25$,
    such that given integer $k \neq 9$ and $8 \le k \le c$,
    for any integer $p \ge k + 2$ and $q \ge k + 1$, $r_k(p,q) \ge (p-1) \cdot (r_k(p-1,q)-1) + 1$ holds.
\end{theorem}

The difference between Theorem~\ref{k567_p-1} and \ref{k_p-1} is the base cases of $q$, which are $k+2$ and $k+1$ respectively. Note that the right-hand side of the recurrence relation in Theorem~\ref{k_p-1} on initial values is $r_k(k+1,k+1)$: the first non-trivial Ramsey number on $k$-hypergraphs.

\subsection{Proof Sketch}
Before proving the above theorems, we take a detour to revisit Corollary~\ref{partition_1}.
We show that Statement \ref{statement_1} in Corollary~\ref{partition_1} can be interpreted in a slightly different way.

\begin{lemma}{\label{partition_2}}
    Define $\mathbf{P_p}(d) = \{\mathbf{v} \mid \sum_{i \in [\pi(v)]} v_i = p, \pi(\mathbf{v}) \le d, v_1 \le p-2 \}$.
    Define $\mathbf{Q_q}(d) = \{\mathbf{v} \mid \sum_{i \in [\pi(v)]} v_i = q, \pi(\mathbf{v}) \le d, v_1 \le q-1 \}$.
    Given integers $p$, $q$, $k$, $d$, $V = \bigcup_{i \in [d]} V_i$, and $\mathbf{G} = V^{(k)}$ as before, the following two statements are equivalent:
    \begin{enumerate}
        \item $\exists \chi$ such that $\forall \mathbf{v} \in \mathbf{P_p}(d)$, $\exists \mathbf{v'} \in \mathbf{V_k}(d)$ such that $\mathbf{v'} \le_c \mathbf{v}$ and $\chi(\mathbf{v'}) = \text{blue}$.
            Moreover, $\chi$ also satisfies that $\forall \mathbf{v} \in \mathbf{Q_q}(d)$, $\exists \mathbf{v'} \in \mathbf{V_k}(d)$ such that $\mathbf{v'} \le_c \mathbf{v}$ and $\chi(\mathbf{v'})= \text{red}$. \label{4.1_statement_1}
        \item $\exists \chi$ such that $\forall p$-subset (resp. $q$-subset) $X \subseteq V$, $\exists k$-hyperedge $e$ of $X^{(k)}$ such that $\chi(e)= \text{blue}$ (resp. red).
    \end{enumerate}
\end{lemma}
\begin{proof}
    Given $p$-subset $X \subseteq V$, if $\exists i \in [d]$ such that $|X \wedge V_i| \ge p - 1$, $X^{(k)}$ must contain a blue edge, because $\chi_i$ is a $(p-1,q;k)$-coloring on ${V_i}^{(k)}$ and $X^{(k)}$ contains some $(p-1)$-clique, which cannot be a red clique.
    Analogously, any $q$-subset $Y$ intersecting with any $V_i$ on more than $q-1$ vertices necessarily contains a red edge, because any ${V_i}^{(k)}$ has a $(p,q;k)$-coloring.
 \end{proof}

This lemma enables us to consider only a proper subset of the previous primal cardinality vector set, leading to a simpler proof of our theorems.
We give a simple proof of Theorem~\ref{k4}, and we prove Theorem~\ref{k567_p-1} and \ref{k_p-1} in the next subsection.
\begin{proof}[Proof of Theorem~\ref{k4}]
    Firstly we prove that for any integer $p \ge 6$ and $q \ge 5$, $r_4(p,q) \ge 2 r_4(p-1,q) - 1$.
    By Lemma~\ref{constant}, it is sufficient to prove that $r_4(6,5) \ge 2 (r_4(5,5) - 1) + 1$.
    We give a $(6,5;4)$-coloring as follows:
    \begin{equation*}
        \chi_1^{(4)} = \{\chi(3,1) = \text{red}, \chi(2,2) = \text{blue}\}.
    \end{equation*}
    To prove $\chi_1^{(4)}$ is a $(6,5;4)$-coloring, by  Lemma~\ref{partition_2}, we need to check the following:
    \begin{itemize}
        \item $\forall \mathbf{v} \in \mathbf{P_6}(2), \exists \mathbf{v'} \le_c \mathbf{v}$, such that $\chi_1^{(4)}(\mathbf{v'}) = \text{blue}$.
            This is true because $\mathbf{P_6}(2) = \{(4,2), (3,3)\}$, both $\ge_c (2,2)$.
        \item $\forall \mathbf{v} \in \mathbf{Q_5}(2), \exists \mathbf{v'} \le_c \mathbf{v}$, such that $\chi_1^{(4)}(\mathbf{v'}) = \text{red}$.
            Since $\mathbf{Q_5}(2) = \{(4,1), (3,2)\}$, each of them $\ge_c (3,1)$.
    \end{itemize}
    Thus we proved $r_4(6,5) \ge 2 r_4(5,5) - 1$.

    Now we need to prove that for any integer $p \ge 6$ and $q \ge 7$, $r_4(p,q) \ge (p-1) (r_4(p-1,q)-1) + 1$
    by starting with proving the case of $p=6, q=7$.
    We give a $(6,7;4)$-coloring as following:
    \begin{equation*}
        \chi_2^{(4)} = \{\chi(3,1) = \chi(1,1,1,1) = \text{red}\}  \cup \{\chi(2,2) = \chi(2,1,1) = \text{blue}\} .
    \end{equation*}
    The following needs to be checked:
    \begin{itemize}
        \item $\forall \mathbf{v} \in \mathbf{P_6}(5)$, we have $2 \le v_1 \le 4$, thus either $(2,2) \le_c \mathbf{v}$ or $(2,1,1) \le_c \mathbf{v}$, which are blue.
        \item $\forall \mathbf{v} \in \mathbf{Q_7}(5)$, it must be that either $v_1 \ge 3$ and $2 \le \pi(\mathbf{v}) \le 3$ or $\pi(\mathbf{v}) \ge 4$. The first case $\ge_c (3,1)$ and the second case $\ge_c (1,1,1,1)$, which are both red.
    \end{itemize}
    By the same reasoning, one can show that $\chi_2^{(4)}$ is also a $(7,7;4)$-coloring and an $(8,7;4)$-coloring. Since now the recurrence relation holds for $p=8, q=7$, we can apply Lemma~\ref{dp} to get $\forall p \ge 8, q \ge 7,r_4(p,q) \ge (p-1) \cdot (r_4(p-1,q)-1) + 1$.
    Combining all these cases we proved the theorem.
 \end{proof}

\subsection{Automated Theorem Proving}

The ``$\exists\forall$'' structure of Statement \ref{4.1_statement_1} in Lemma~\ref{partition_2} reminds us of Propositional Logic Satisfiablity (SAT). In fact, a $(p,q;k)$-coloring $\chi$ serves as a certificate of the proof for theorem $r_k(p,q) \ge d\cdot (r_k(p-1,q) - 1) + 1$.
Thus it is nature to use automated theorem proving instead of proving it by hand. As we saw in the proof of Theorem~\ref{k4}, even the simplest case is time-consuming to verify, regardless of how to find that coloring.

\begin{definition}\label{CNF_SAT}
    A Conjunctive Normal Form (CNF) is a conjunction of clauses, such that each clause is a disjunction of literals, where a literal can be positive of negative variable. A satisfying assignment of CNF is a mapping from all variables to \emph{true} or \emph{false} such that every clause has at least one true literal.
    A SAT solver takes a CNF as input and outputs a satisfying assignment or UNSAT if the CNF is unsatisfiable.
\end{definition}

We give the procedure to prove $r_k(p,q) \ge d\cdot (r_k(p-1,q) - 1) + 1$ for fixed $p,q$, then Lemma~\ref{constant} and \ref{dp} can be applied to prove it for arbitrary $p,q$:

\begin{enumerate}
    \item For every $\mathbf{v} \in \mathbf{P_p}(d)$, construct a clause $C_p(\mathbf{v})$ as follow: For every $\mathbf{u} \in \mathbf{V_k}(d)$, if $\mathbf{u} \le_c \mathbf{v}$, add a positive variable $x(\mathbf{u})$ in $C_p(\mathbf{v})$.
    \item For every $\mathbf{v} \in \mathbf{Q_q}(d)$, construct a clause $C_q(\mathbf{v})$ as follow: For every $\mathbf{u} \in \mathbf{V_k}(d)$, if $\mathbf{u} \le_c \mathbf{v}$, add a negative variable $\neg x(\mathbf{u})$ in $C_q(\mathbf{v})$.
    \item Use SAT solver to solve the constructed CNF:
    \begin{equation*}
        F = \left(\bigcup_{\mathbf{v} \in \mathbf{P_p}(d)} C_p(\mathbf{v})\right) \bigcup \left(\bigcup_{\mathbf{v} \in \mathbf{Q_q}(d)} C_q(\mathbf{v}) \right).
    \end{equation*}
    \item If a satisfying assignment $\alpha$ is found, we construct a $(p,q;k)$-coloring $\chi$ as follows:
        if $\alpha(x(\mathbf{u})) = \text{true}$, set $\chi(\mathbf{u}) \coloneqq \text{blue}$;
        if $\alpha(x(\mathbf{u})) = \text{false}$, set $\chi(\mathbf{u}) \coloneqq \text{red}$.
\end{enumerate}

It is easy to show that this procedure is a correct proof when SAT solver returns a satisfying assignment: $\forall \mathbf{v} \in \mathbf{P_p}(d)$, $\exists \mathbf{u} \in \mathbf{V_k}(d)$ such that $\mathbf{u} \le_c \mathbf{v}$ and $\chi(\mathbf{u}) = \text{blue}$, because $\exists x(\mathbf{u}) \in C_p(\mathbf{v})$ such that $x(\mathbf{u}) = \text{true}$;
similarly, $\forall \mathbf{v} \in \mathbf{Q_q}(d)$, $\exists \mathbf{u} \in \mathbf{V_k}(d)$ such that $\mathbf{u} \le_c \mathbf{v}$ and $\chi(\mathbf{u}) = \text{red}$, because $\exists x(\mathbf{u}) \in C_q(\mathbf{v})$ such that $x(\mathbf{u}) = \text{false}$.
So by Lemma~\ref{partition_2} we proved the recurrence relation holds for $p$ and $q$.

\begin{proof}[Proof of Theorem~\ref{k567_p-1} and Theorem~\ref{k_p-1}]
    We use the latest version of SAT solver from \cite{DBLP:conf/aaai/LiuP16} to solve the following two kinds of CNFs:
    \begin{align*}
        F_1 &= \left(\bigcup_{\mathbf{v} \in \mathbf{P_{k+2}}(k+1)} C_p(\mathbf{v})\right) \bigcup \left(\bigcup_{\mathbf{v} \in \mathbf{Q_{k+2}}(k+1)} C_q(\mathbf{v}) \right). \\
        F_2 &= \left(\bigcup_{\mathbf{v} \in \mathbf{P_{k+3}}(k+2)} C_p(\mathbf{v})\right) \bigcup \left(\bigcup_{\mathbf{v} \in \mathbf{Q_{k+2}}(k+2)} C_q(\mathbf{v}) \right).
    \end{align*}

    Our SAT solver returns satisfying assignments on all $5 \le k \le 25$. The satisfying assignment of $F_1$ is a proof for the recurrence relation of case $p = k + 2$ and $q = k + 2$, while that of $F_2$ is a proof for the case $p = k + 3$ and $q = k + 2$. Therefore, by Lemma~\ref{dp} we proved Theorem~\ref{k567_p-1}.

    We do the same for the CNF corresponding to $p=k+2, q=k+1$ on all $8 \le k \le 25$, and get satisfying assignments on all $k$ except for $k = 9$ returning UNSAT, thus (with Lemma~\ref{dp}) proved Theorem~\ref{k_p-1}.
 \end{proof}

Given more time on constructing more CNFs on larger $k$, it is almost sure that lower bound for $c$ in Theorem~\ref{k_p-1} can be improved.
As a result, we give the following conjecture as the $c$-unbounded version of Theorem~\ref{k_p-1}.
\begin{conjecture}{\label{any_k_p-1}}
    Given integer $k \ge 10$,
    for any integer $p \ge k + 2$ and $q \ge k + 1$, $r_k(p,q) \ge (p-1) \cdot (r_k(p-1,q)-1) + 1$.
\end{conjecture}

\section{Recurrences for Arbitrary $\mathbf{k}$}\label{bigk}
In this section, we give two recurrence relations for arbitrary $k$.
The recurrence forms align with forms (2) and (3) in \S{\ref{forms}}.
\begin{theorem}\label{two_divide}
    Given even integer $k \ge 4$, for any integers $p \ge k + 2, q \ge k + 1$, $r_k(p, q) \ge 2 \cdot (r_k(p - 1, q) - 1) + 1$ holds.
    Given odd integer $k \ge 5$, for any integers $p \ge k + 2, q \ge k + 2$, the same recurrence relation holds.
\end{theorem}
\begin{proof}
    By Lemma~\ref{constant}, it is sufficient to prove
    that $r_{k}(k+2, k+1) \ge 2 r_k(k+1, k+1) - 1$ holds for even $k \ge 4$
    and
    $r_{k}(k+2, k+2) \ge 2 r_k(k+1, k+2) - 1$ holds for odd $k \ge 5$.
    The rest of the proof is an induction on $k$.
    For $k = 4$, by Theorem~\ref{k4}, the recurrence relation holds.
    The case of $k = 5$ is implied by Theorem~\ref{k567_p-1}. Assuming the recurrence holds for $k$, we prove the inductive step on $k+1$.

    First we deal with the case where $k$ is even.
    We need to prove that
    $r_{k + 1}(k + 3, k + 3) \ge 2 r_{k + 1}(k + 2, k + 3) - 1$.
    The proof is by constructing a coloring $\chi^{(k + 1)}$ satisfying Statement \ref{4.1_statement_1} of Lemma~\ref{partition_2}.
    The coloring $\chi^{(k + 1)}$ is defined as following:
    $$\forall (u_1, u_2) \in \mathbf{V_{k + 1}}(2), ~\chi^{(k + 1)}(u_1, u_2) = \chi^{(k)}(u_1 - 1, u_2).$$

    Since $u_1 + u_2 = k + 1$ is odd, it must be $u_1 \ge u_2 + 1$, thus $\chi^{(k)}(u_1 - 1, u_2)$ is defined.
    We need to show that the following two conditions hold:
    (\romannumeral1) $\forall (v_1, v_2) \in \mathbf{P_{k+3}}(2)$, $\exists (u_1, u_2) \in \mathbf{V_{k+1}}(2)$ such that $(v_1, v_2) \ge_c (u_1, u_2)$ and $\chi^{(k + 1)}(u_1, u_2)=\text{blue}$;
    (\romannumeral2) $\forall (v_1, v_2) \in \mathbf{Q_{k+2}}(2)$, $\exists (u_1, u_2) \in \mathbf{V_{k+1}}(2)$ such that $(v_1, v_2) \ge_c (u_1, u_2)$ and $\chi^{(k + 1)}(u_1, u_2)=\text{red}$.

    We prove condition (\romannumeral1) holds first.
    Since $v_1 + v_2 = k+3$ is odd, $v_1 \ge v_2 + 1$. Observe that $(v_1 - 1, v_2) \in \mathbf{P_{k+2}}(2)$.
    In the inductive step, it is assumed that the case of $k$ holds, we have that $\forall (v_1', v_2') \in \mathbf{P_{k+2}}(2), \exists (u_1', u_2') \in \mathbf{V_k}(2)$ such that $(v_1', v_2') \ge_c (u_1', u_2')$ and $\chi^{(k)}(\mathbf{u}) = \text{blue}$.
    Now let $v_1 - 1 = v_1', v_2 = v_2'$.
    Since $(v_1 - 1, v_2) \ge_c (u_1', u_2')$, we have that $(v_1, v_2) \ge_c (u_1' + 1, u_2')$. By definition, $\chi^{(k + 1)}(u_1' + 1, u_2') = \chi^{(k)}(u_1', u_2') = \text{blue}$ and $(v_1, v_2)$ contains a blue edge.

    We also need to show that condition (\romannumeral2) holds.
    By the same reasoning $\forall (v_1, v_2) \in \mathbf{Q_{k+3}}(2)$, $\exists (u_1, u_2) \in \mathbf{V_{k+1}}(2)$ such that $(v_1, v_2) \ge_c (u_1, u_2)$ and $\chi^{(k + 1)}(u_1, u_2)=\text{red}$.

    Finally, we prove the case where $k$ is odd.
    We need to prove that
    $r_{k + 1}(k + 3, k + 2) \ge 2 r_{k + 1}(k + 2, k + 2) - 1$.
    $\forall (v_1, v_2) \in \mathbf{P_{k+3}}(2)$, if $v_1 \ge v_2 + 1$, analogy to the even case, $(v_1,v_2)$ must contain a blue edge. But when $v_1 = v_2 = \frac{k+3}{2}$, $(v_1, v_2)$ contains edges $(\frac{k+3}{2}, \frac{k-1}{2})$ and $(\frac{k+1}{2}, \frac{k+1}{2})$. Note that $\chi^{(k+1)}(\frac{k+1}{2}, \frac{k+1}{2})$ is not defined yet, because $(\frac{k+1}{2} - 1, \frac{k+1}{2})$ is not valid. Thus we assign $\chi^{(k+1)}(\frac{k+1}{2}, \frac{k+1}{2})$ to a color different from $\chi^{(k+1)}(\frac{k+3}{2}, \frac{k-1}{2})$. So $(v_1, v_2)$ still contains a blue edge.
    Similarly, $\forall (v_1, v_2) \in \mathbf{Q_{k+2}}(2)$, we have $v_1 \ge v_2 + 1$ because $v_1 + v_2 = k + 2$ is odd. Then by the same reasoning in the even $k$ case, $(v_1, v_2)$ must contain a red edge.
\end{proof}

\begin{theorem}\label{large_q}
    Given any integer $k \ge 4$, for any integers $p \ge k + 2, q \ge k + 1$, $r_k(p, q) \ge d \cdot (r_k(p - 1, q) - 1) + 1$ holds, where $d = \lfloor \frac{q - 1}{k - 2} \rfloor$.
\end{theorem}
\begin{proof}
    If $d \le 2$, this is implied by Theorem~\ref{two_divide}.
    Now assume $d \ge 3$.
    Define coloring as follows: $\chi^{(k)}(\mathbf{v}) = \text{red}$ if and only if $\mathbf{v} = (k-1, 1)$; otherwise $\chi^{(k)}(\mathbf{v}) = \text{blue}$. We show that under such $\chi^{(k)}$, $\forall \mathbf{v} \in \mathbf{P_p}(d)$, $\exists \mathbf{v'} \in \mathbf{V_k}(d)$ such that $\mathbf{v'} \le_c \mathbf{v}$ and $\chi^{(k)}(\mathbf{v'})=\text{blue}$;
    and $\forall \mathbf{v} \in \mathbf{Q_q}(d)$, $\exists \mathbf{v'} \in \mathbf{V_k}(d)$ such that $\mathbf{v'} \le_c \mathbf{v}$ and $\chi^{(k)}(\mathbf{v'})= \text{red}$.

    Firstly, for any $\mathbf{v} \in \mathbf{P_p}(d)$, there are two cases. The first case is that if $\pi(\mathbf{v}) \ge 3$, then $\exists \mathbf{v'} \in \mathbf{V_k}(d)$, such that $\pi(\mathbf{v'}) \ge \min\{k, \pi(\mathbf{v})\}$ and $\mathbf{v} \ge_c \mathbf{v'}$, so $\chi^{(k)} = \text{blue}$ since $\min\{k, \pi(\mathbf{v})\} \ge 3$. The existence of such $\mathbf{v'}$ can be proved by the following process:
    Initialize $\mathbf{v'}$ as $\mathbf{v}$.
    If $\exists i \in [\pi(\mathbf{v'})-1], v'_i > v'_{i+1}$, let $\mathbf{v''} \coloneqq \mathbf{v'} - \mathbf{e_i}$; 
    else let $\mathbf{v''} \coloneqq \mathbf{v'} - \mathbf{e_{\pi(\mathbf{v'})}}$. 
    Clearly $\mathbf{v'} \ge_c \mathbf{v''}$ and $\mathbf{v''} \in \mathbf{P_{p-1}}(d')$ where $d' = d$ or $d-1$.
    Update $\mathbf{v'}$ to $\mathbf{v''}$, $p$ to $p-1$, and $d$ to $d'$, then repeat the above.
    We do this until $p$ reaches $k$. 
    If $\pi(\mathbf{v}) \ge k$ then $v'_{k} \ge 1$, else $v'_{\pi(\mathbf{v})} \ge 1$. Both of them lead to $\pi(\mathbf{v'}) \ge \min\{k, \pi(\mathbf{v})\}$ and $\mathbf{v} \ge_c \mathbf{v'}$ by transitivity.
    The second case is that $\pi(\mathbf{v}) = 2$.
    This is straightforward since $v_1 \le p-2$ (Lemma~\ref{partition_2}), it must be $v_2 \ge 2$. Just let $u_2 = \max(v_2 - \lceil \frac{p - k}{2} \rceil, 2)$ and $u_1 = k - u_2$, we have $u_1 \ge u_2 \ge 2$ and $\chi^{(k)}(\mathbf{u}) = \text{blue}$  by definition.

    Secondly, $\forall \mathbf{v} \in \mathbf{Q_q}(d)$, by the Pigeonhole principle, it must be $v_1 \ge \lfloor \frac{q - 1}{d} \rfloor + 1 = k - 1$. Additionally, by Lemma~\ref{partition_2} we have $v_1 \le q - 1$, thus $v_2 \ge 1$, so $(v_1, v_2) \ge_c (k-1, 1)$. Since $\chi^{(k)}(k-1, 1) = \text{red}$ by definition, we proved that $\forall \mathbf{v} \in \mathbf{Q_q}(d)$, $\exists u \in \mathbf{V_k}(d)$, such that $\mathbf{v} \ge_c \mathbf{u}$ and $\chi^{(k)}(\mathbf{u}) = \text{red}$.

    Combining both we proved the theorem.
 \end{proof}

\section{Improved Lower Bounds}\label{lower_bounds}
We summarize some of our improved lower bounds for Ramsey numbers on hypergraphs in this section.

\subsection{$4$-hypergraph}

Previous best lower bounds for Ramsey number on $4$-hypergraphs can be found in \cite{DBLP:journals/dm/SongYL95} and \cite{radziszowski1994small}. We point out that some of their values are based on \cite{DBLP:journals/dm/Shastri90} whose calculation of $r_4(7,7)$ is wrong. The following lower bounds values in the ``Previous'' column are re-calculated in a corrected way using their methods.
We also add a ``Reference'' column for the method we use to derive our results. Some representative results are displayed below.

\begin{table}[h]
\centering\renewcommand{\arraystretch}{1.2}
\begin{tabular}{|c|c|c|c|}\hline
 & Previous & Our Result & Reference \\\hline
 $r_4(5,6) \ge$ & 37 & \textbf{67} & Theorem \ref{k4} or \ref{two_divide} \\\hline
 $r_4(6,6) \ge$ & 73 & \textbf{133} & Theorem \ref{two_divide} or \ref{large_q} \\\hline
 $r_4(6,7) \ge$ & 361 & \textbf{661} & Theorem \ref{k567_p-1} \\\hline
 $r_4(6,13) \ge$ & 23041 & \textbf{50689} & Theorem \ref{large_q} \\\hline
 $r_4(7,7) \ge$ & 2161 & \textbf{3961} & Theorem \ref{k567_p-1} \\\hline
 $r_4(8,8) \ge$ & 105841 & \textbf{194041} & Theorem \ref{k567_p-1} \\\hline
\end{tabular}
\end{table}

Using our Algorithm \ref{new_alg}, a coloring for proving $r_4(5,5) \ge 34$ can be found (see Appendix~\ref{local_search}).
The subsequent lower bounds can be obtained using the corresponding recurrence relations.

\subsection{$5$-hypergraph}
Before this work, there is no constructive lower bounds for Ramsey numbers on $5$-hypergraphs.

Using our Algorithm \ref{new_alg}, a coloring for proving $r_5(6,6) \ge 72$ can be found, which serves as a certificate of the lower bound.
Subsequently, lower bounds for $r_5(p, q)$ can be calculated using our Theorems \ref{k_p-1} and \ref{large_q}.

\subsection{$\ge 6$-hypergraph}
Previously, there is neither constructive nor recursive lower bounds for Ramsey number on $\ge 6$-hypergraphs.

The base case of the recurrence relation is
$r_k(k+1, k+1) \ge r_k(k + 1, k) = k + 1$.
For any $k \ge 6$, lower bounds for $r_k(p, q)$ can be calculated using our Theorems \ref{k567_p-1}, \ref{k_p-1}, \ref{two_divide} and \ref{large_q}.

\paragraph{Acknowledgements.}
Research at Princeton University partially supported by an innovation research grant from Princeton and a gift from Microsoft.

%
%
%
\bibliographystyle{alpha}
\bibliography{ramsey}
%

\appendix

\section{The Constructive Algorithm}\label{local_search}

For $k$-hypergraphs, the first non-trivial Ramsey number $r_k(k + 1, k + 1)$  is at least $r_k(k + 1, k) = k + 1$, which serves as a lower bound for $r_k(k + 1, k + 1)$, and can be fed to our recurrence relation to derive lower bounds for all $r_k(p,q)$. But note that there is a straightforward CNF encoding for proving $r_k(k + 1, k + 1) > N$, thus using SAT solver on it might do much better:
Each $k$-hyperedge corresponds to a variable; each $(k + 1)$-clique corresponds to two clauses: the first clause contains $k+1$ positive literals, thus must contain a true (blue) variable; the second clause contains $k+1$ negative literals, thus must contain a false (red) variable. If CNF corresponding to $k$-hypergraph on $N$ vertices has a solution, we explicitly found a good coloring, which is a proof.

In this section, we propose a new algorithm for proving $r_k(k + 1, k + 1) > N$ based on local search. We give necessary definitions in our local search algorithm, followed by the algorithmic framework.
Note that some previous lower bounds are obtained by fairly complex Genetic algorithms \cite{exoo1989lower}, while our local search algorithm is simple and easy to implement.

\subsection{Notations in Local Search}
We need some additional notations.
A satisfied clauses has at least $1$ true literal, and a $2$-satisfied clause has at least $2$ true literals. 
We use $V(F)$ to denote all variables appear in CNF $F$.
\begin{definition}
    Given CNF $F$ and a truth assignment $\alpha$ on $V(F)$, $\forall v \in V(F)$, define $\text{score}(v)$ as the increment in number of satisfied clauses after flipping $v$, and define $\text{subscore}(v)$ as the increment in number of $2$-satisfied clauses after flipping $v$, and define $\text{age}(v)$ as the number of flips performed since the last flip of $v$.
\end{definition}

In local search for SAT, the goal is to minimize the number of unsatisfied clause, thus intuitively one should prefer to flip variable with greater $\text{score}$ and greater $\text{subscore}$. However as we will show later in our algorithm, to solve CNF corresponding to theorem proving, one should prefer variable with smaller $\text{subscore}$.

One vital problem in local search is to deal with local optimal, i.e., a point in the solution space with no better point nearby to move to.
Two influential strategies in the literature are:
(\romannumeral1) \emph{Tabu} \cite{DBLP:conf/aaai/MazureSG97}: variables with age fewer than a preset threshold are forbidden to flip,
(\romannumeral2) \emph{Configuration checking} \cite{DBLP:conf/aaai/CaiS12}: variables appearing in the same clause are called \emph{neighborhood} to each other, variables with no neighborhoods flipped since its last flip are forbidden to flip.
Unfortunately, both Tabu and Configuration checking fail in theorem proving, because Tabu ignores local structure and Configuration checking never forbids any variable dues to the fact that all variables are neighborhoods to each other (any two hyperedges appear in some clique).

Our key observation is that in a hypergraph, two hyperedges can either share endpoints or not, therefore a mechanism called \emph{Neighborhood checking} can be defined as follows.
\begin{definition}
    Given CNF $F$ corresponding to $r_k(k + 1, k + 1) > N$, $\forall v_1, v_2 \in V(F)$, if the hyperedge corresponding to $v_1$ shares endpoints with the hyperedge corresponding to $v_2$, then $v_1$ is called a \emph{neighborhood} of $v_2$.
    Define \emph{neighborhood checking} $\text{nc}(v)$ as an indicator of whether any neighborhood of $v$ has been flipped since the last flip of $v$.
\end{definition}

The updating rules for $\text{nc}(v)$ is straightforward: when $v$ is flipped, set $\text{nc}(v) \coloneqq \text{False}$ and for each of the neighborhood $u$ of $v$, set $\text{nc}(u) \coloneqq \text{True}$.

\subsection{Algorithmic Framework}\label{new_alg}
First we define the tie-breaking function $H$ used in our algorithm: return variable with the greatest $\text{score}$; if ties, return one with the \emph{smallest} $\text{subscore}$; if still ties, return one with the greatest $\text{age}$.
We give the algorithmic framework as follows:

\begin{enumerate}
    \item Input CNF $F$ corresponding to $r_k(k + 1, k + 1) > N$.
    \item Generate a uniform random assignment on all variables.
    \item Repeat the following until all clauses are satisfied or a preset number of flips is reached (cutoff):
        \begin{itemize}
            \item If there exists variable $v$ with $\text{nc}(v) = \text{True}$ and $\text{score}(v) > 0$, then flip $v$. If there are more than $1$ such variables, breaks tie using function $H$.
            \item Else if no such variable exists, choose an unsatisfied clause $c$ uniformly at random. Choose the best variable in $c$ according to function $H$, then flip it.
        \end{itemize}
    \item If all clause are satisfied, output the satisfying truth assignment. Else if the cutoff is reached, go to Step 2.
\end{enumerate}


The implementation of the algorithm can be found in \url{https://github.com/sixueliu/RamseyNumber}.
This link also contains two hypergraphs found by our algorithm:
(\romannumeral1) a $33$-vertex $4$-hypergraph with no $5$-clique nor $5$-independent-set, i.e., a proof for $r_4(5,5) \ge 34$;
(\romannumeral2) a $71$-vertex $5$-hypergraph with no $6$-clique nor $6$-independent-set, i.e., a proof for $r_5(6,6) \ge 72$.

\end{document}